\documentclass[11pt]{amsart}
\usepackage[colorlinks,linkcolor=blue,citecolor=blue,urlcolor=red]{hyperref}
\usepackage{amssymb,latexsym,amscd,verbatim,color}
\usepackage[all]{xy}

\usepackage{color}

\swapnumbers
\newtheorem{thm}{Theorem}[subsection]

\newtheorem{lemma}[thm]{Lemma}

\theoremstyle{definition}
\newtheorem{defn}[thm]{Definition}

\newtheorem{remark}[thm]{Remark}

\newcommand{\Q}{\mathbb{Q}}
\newcommand{\Z}{\mathbb{Z}}
\newcommand{\C}{\mathbb{C}}
\newcommand{\T}{\mathbb{T}}
\newcommand{\R}{\mathbb{R}}

\newcommand{\unit}{\mathbf{1}}
\newcommand{\cC}{\mathcal{C}}

\newcommand{\cA}{\mathcal{A}}
\newcommand{\cB}{\mathcal{B}}
\newcommand{\cM}{\mathcal{M}}
\newcommand{\cE}{\mathcal{E}}

\newcommand{\tensor}{\otimes}

\newcommand{\Ind}{\mathrm{Ind}}
\newcommand{\id}{\mathrm{id}}

\newcommand{\catF}[1]{\mathrm{Ab}(#1)}
\newcommand{\catN}[1]{\mathrm{\mathcal{A}}(#1)}

\newcommand{\Rmod}{\text{$\R$-{\sf Mod}}}

\newcommand{\Nori}{\mathrm{Nori}}
\newcommand{\good}{\mathrm{good}}

\title{Tensor product of motives via K\"unneth formula}
\author{Luca Barbieri-Viale}
\address{Dipartimento di Matematica ``F. Enriques", Universit{\`a} degli Studi di Milano\\ Via C. Saldini, 50\\ I-20133 Milano\\ Italy}
\email{luca.barbieri-viale@unimi.it}
\author{Mike Prest}
\address{Department of Mathematics, Alan Turing Building, University of Manchester, Manchester M13 9PL, UK}
\email{mprest@manchester.ac.uk}
\begin{document}

\begin{abstract}
Following Nori's original idea we here provide certain motivic categories with a canonical tensor structure. These motivic categories are associated to a cohomological functor on a suitable base category and the tensor structure is induced by the cartesian tensor structure on the base category via a cohomological K\"unneth formula.
\end{abstract}
\maketitle
\section*{Introduction}
We here make use of Nori's approach to motives by considering the universal abelian category ${\sf ECM}_\cC^H$  corresponding to a cohomology $H$ on a given category $\cC$ which is endowed with some good geometric properties, as reformulated and generalised in \cite{BV} and \cite{BVP}. Making use of \cite{BVHP} we show that a suitable K\"unneth formula for $H$, along with a cellularity condition, provide ${\sf ECM}_\cC^H$ with a canonical tensor structure: this is our main theorem, see Theorem~\ref{thm:Tmot} below. In particular, this framework applies to Nori's original category of motives and to several other geometric situations, e.g., it applies to \cite{Ar}, \cite{I1}, \cite{I2},  \cite{FJ} and \cite{NP}. Recall that Nori's category, i.e., the category of (effective,  cohomological) Nori motives over a subfield $k$ of the complex numbers, is obtained in this way as ${\sf ECM}_\cC^H$ for $H = H_{\rm sing}$  singular cohomology and $\cC$ being the category of (affine) schemes over the field $k$ (e.g., see \cite{HMS} for details on Nori's original construction and \cite{BCL} for its reconstruction).  

\subsection*{Cohomology theories} Following \cite{BV}, a cohomology theory on a given (small) category $\cC$, regarded as a category of geometric objects, can be increasingly specified by imposing conditions which reflect the geometric properties of $\cC$ and which we want to hold true for those $H$ that we wish to regard as cohomologies.  Some of these conditions may be expressed as axioms within the framework of a formal language.  As a starting point, a cohomology theory should express the basic idea of associating $(X,Y)\leadsto \{H^n (X,Y)\}_{n\in\Z}$ in such a way that 
\begin{itemize}
\item[{\it i)}] $H = \{H^n\}_{n\in\Z}$ is a family of contravariant functors  $H^n : \cC^{\square}\to \cA$  on a suitable category of pairs $\cC^\square$, taking values in an abelian category $\cA$,
\item[{\it ii)}] for the triple $(X,Y)$, $(X,Z)$ and $(Y,Z)$ we have an extra connecting morphism $\delta^n : H^n (Y,Z)\to H^{n+1} (X,Y)$, and
\item[{\it iii)}] we have the following long exact sequence
\begin{equation}\label{longexact}
\cdots\to H^n (X,Y)\to H^n (X,Z)\to H^n (Y,Z)\to H^{n+1} (X,Y)\to \cdots
\end{equation}
which is also assumed to be natural in a canonical way (see \cite[\S 3.1]{BV} for more details). 
\end{itemize} 

Note that we may weaken the additivity assumption on $\cA$, as we may appeal to the notion of exact category (in the sense of Barr \cite{Ba}) where we still can talk of exact sequences of pointed objects. For example, $\cA$ can be the category of sets or that of groups, these being Barr-exact. In fact, to express this basic idea of a cohomology $H: \cC^{\square}\to \cA$, the assumptions on $\cA$ can be weakened further: we just need a regular category (see also \cite{Ba} for this notion), i.e., a category with finite limits where every arrow factorises as a regular epi followed by a mono and these factorisations are stable under pullback. Recall that regular epis are just coequalizers of some parallel pair of morphisms, and, in a regular category, the regular epi-mono factorisation is unique, so we can talk of the image of a morphism. 

There is a general key relation between regular categories and logic, which we outline (see, e.g., \cite{Bu} for more detail). 

The first important fact is that regular categories are well suited for the interpretation of any regular theory. A regular theory $\R$ is a set of sequents $\top\vdash_{x,y} \varphi \rightarrow  \psi$, also called axioms, where $\varphi$ and $\psi$ are regular formulas involving sorts for which $x$ and $y$ are terms and/or variables. The regular formulas (also called positive primitive formulas) are those (first order) formulas built up from the atomic formulas (= equations, in any algebraic theory) using only meet $\wedge$ (``and") and existential quantification $\exists$. For example, to express exactness in the long exact sequence \eqref{longexact} displayed above we may include the axioms $\top\vdash_{y} fg(y) = * $ and $\top\vdash_{x}  f (x) = * \rightarrow (\exists y) g (y)= x$ for any composable function symbols $g$ and $f$ and where $*$ is the constant of each sort ($0$ in the additive case). In the case of an abelian category, and using an alternative notation, these axioms would read $\forall y\,\, fg(y)=0$ and $\forall x \, (f(x)=0 \rightarrow \exists y (g(y)=x))$.

Now, given a regular theory $\R$, we can interpret it in any regular category $\cA$; any such interpretation in which every sequent in $\R$ is valid is said to be a model of the theory (e.g., see \cite[\S 3]{Bu}). A functor $F: \cA\to \cB$ between regular categories is called exact (or regular) if it preserves finite limits and regular epis. It follows that models are preserved by exact functors.   We can form the category of models $\Rmod (\cA)$ of $\R$ in $\cA$, and this gives a functor $\Rmod ( - )$ on the category of regular categories with exact functors. The second important fact is that this functor is representable.  That is, given $\R$, there exists a regular category $\cC_{\R}^{\rm reg}$ and an equivalence between the models in a regular category $\cA$ and the exact functors from $\cC_{\R}^{\rm reg}$ to $\cA$, i.e., we have
\begin{equation}\label{repreg}
\Rmod (\cA) \cong {\sf Ex} (\cC_{\R}^{\rm reg}, \cA )
\end{equation}
which is natural in $\cA$ (e.g., see \cite[\S 6]{Bu} for details). The category $\cC_{\R}^{\rm reg}$ is the so-called regular syntactic category for $\R$ and the model $U$ in  $\Rmod (\cC_{\R}^{\rm reg})$ corresponding to the identity functor on $\cC_{\R}^{\rm reg}$  is called the universal (or generic) model. A third fact is that any regular category is the regular syntactic category of some regular theory (see \cite[\S 5]{Bu}). 

Finally, let $\cC_{\R}^{\rm reg/ex}$ be the canonical (Barr-)exact completion (e.g., see \cite{La}) of the syntactic category; so
\begin{equation}\label{regex}
{\sf Ex} (\cC_{\R}^{\rm reg}, \cA )\cong {\sf Ex} (\cC_{\R}^{\rm reg/ex}, \cA )
\end{equation}
whenever $\cA$ is exact. Note that the equivalence in \eqref{regex} is natural in $\cA$ (and implicitly defines what we mean by exact completion).  Since the exact completion preserves additivity we have that $\cC_{\R}^{\rm reg}$ additive yields $\cC_{\R}^{\rm reg/ex}$ abelian. Indeed, Tierney's theorem is that an exact additive category $\cA$ is abelian and conversely.

These are the main facts arising from the general theory of regular categories that we will apply to the study of cohomology. In \cite[\S 2]{BV}, for a fixed base category $\cC$, a {\em regular cohomology theory} $\T^{op}$ consists of the initial set of regular axioms, which express the properties {\it i) -- iii)}\, on any 
$H = \{H^n\}_{n\in \Z}$ as indicated above: any such specific cohomology $H$ is then referred to as a model of the theory $\T^{op}$; that is $H$ is an interpretation of the axioms in some category $\cA$ which should at least be regular. So, in the following, saying that $H$ is a \emph{$\T^{op}$-model} is synonymous with saying that $H$ is a \emph{cohomology}. 

\subsection*{Effective cohomological motives} Therefore, following \cite{BV}, for a fixed category $\cC$ a category of {\em effective constructible $\T^{op}$-motives} can be defined as follows
\begin{equation}\label{deftmot}
\cA[\T^{op}] := \cC_{\T^{op}}^{\rm reg/ex}
\end{equation}
via the exact completion of the regular syntactic category associated to the regular cohomology theory $\T^{op}$ mentioned above.
For any cohomology $H: \cC^{\square}\to \cA$ with $\cA$ exact we get an induced realisation exact functor $r_H : \cA[\T^{op}]\to \cA$ as it follows directly from \eqref{repreg} and \eqref{regex} above. 
Moreover, for any such cohomology $H$ we then get the regular theory $\T_H^{op}$ of that model, a richer cohomology theory obtained by adding to $\T^{op}$ all regular axioms which are valid in $H$, and, therefore, we obtain another universal category 
$\cA[\T_H^{op}]$  (see \cite[\S 3 \& \S 4]{BV},  \cite[\S 2.2]{BVP}, \cite{BCL} and the next \S 1.1 for more details). Note that we always have the larger category of $\T^{op}_H$-motives $\Ind (\cA[\T_H^{op}])$ provided by indization as well as an even larger $\T_H^{op}$-motivic topos $\cE [\T^{op}_H]$ given by the topos of sheaves on $\cA[\T^{op}_H]$ regarded as a site for the regular topology (see \cite{BV}).

Actually, as explained in \cite{BVP}, for any cohomology $H: \cC^{\square}\to \cA$ where $\cA$ is an abelian category, there is a canonical equivalence
\begin{equation}\label{motives}
\cA[\T_H^{op}]\cong \cA (H)
\end{equation}
with the universal abelian category $\cA (H)$ given by $H$ regarded as a representation of Nori's diagram $D^\Nori$ associated to $\cC^\square$ (see \cite[Cor. 2.9]{BVP}). This latter category $\cA (H)$ is obtained by a direct application of Freyd's universal construction (see \cite{Frey}) and Serre localisation. For instance, one has Freyd's universal abelian category ${\rm Ab} (D^\Nori)$ on the preadditive category generated by the diagram $D^\Nori$ and then an exact functor $F_H : {\rm Ab} (D^\Nori)\to \cA$ which is induced by the representation $H$ of $D^\Nori$ in $\cA$; then one obtains $\cA (H)$ as the Serre quotient of ${\rm Ab} (D^\Nori)$ by the kernel of $F_H$. Therefore, following Nori's idea, in the additive case, we may simply refer to both the abelian categories in \eqref{motives} as the \emph{category of (effective, cohomological) motives} associated to $\cC$, a chosen category of geometric objects, and $H$, a paradigmatic cohomology, and we shall denote it by
\begin{equation}\label{Nori}
{\sf ECM}_\cC^H
\end{equation} following Nori's notation. It may well be that a different cohomology $H'$ gives rise to an equivalent category of motives ${\sf ECM}_\cC^H\cong {\sf ECM}_\cC^{H'}$. The ambiguity in the choice of the cohomology is somewhat the motivic analogue of the Tannakian formalism of fibre functors even if we cannot get for free any tensor structure on this category of motives.

\subsection*{Tensor product of motives} In order to get a canonical tensor structure on the category ${\sf ECM}_\cC^H$ we need to appeal to an additional K\"unneth formula for the cohomology $H$, i.e., a K\"unneth formula for the $\T^{op}$-model $H$. First we need to consider the restriction of the cohomology $H$ to {\em good pairs}\, (see Definition~\ref{gp}). We can then express {\em cellularity}\, of $H$ (see Definition~\ref{cell}), a condition which ensures that the cohomology $H$ is determined by its values on good pairs. For the full subcategory $\cC^\good\subseteq \cC^\square$, of good pairs for $H$, we consider $H^\good$, the cohomology $H$ restricted to $\cC^\good$, and we show that there is an equivalence of abelian categories 
\begin{equation}\label{good}
\cA (H)\cong\cA (H^\good)
\end{equation}
as a consequence of the cellularity assumption (see Lemma~\ref{lemma:goodpair}).
For any cohomology $H$ which is {\em cellular and satisfies the K\"unneth formula} (see Definition~\ref{kf}), we thus can provide the abelian category of motives ${\sf ECM}_\cC^H$ with a canonical tensor structure, by making use of the results of \cite{BVHP}, see our Theorem~\ref{thm:Tmot} for the precise statement.  Actually, all this is just following Nori's original idea but the constructions are completely different and fairly general.  The key assumptions on $H$ are made just in order to transfer the cartesian structure of $\cC^\square$ to a canonical tensor structure on $\cA[\T_H^{op}]$ via the cohomology $H$ regarded as a representation of the corresponding subquiver
$D^\good$ of $D^\Nori$, as from \eqref{motives} and \eqref{good} we get
\begin{equation}\label{end}
\cA[\T_H^{op}]\cong \cA (H^\good)
\end{equation}
and $\cA (H^\good)$ has a canonical tensor structure. Note that the cartesian structure of the category $\cC^\good$ of good pairs is obtained by the K\"unneth formula and the tensor structure on $\cA (H^\good)$ is given by $H^\good$ regarded as a $\otimes$-representation of the restricted Nori (graded) $\otimes$-diagram $D^\good$ (see \cite[Thm. 2.18]{BVHP}). Freyd's universal abelian category ${\rm Ab} (D^\good)$ carries an induced right-exact tensor structure which is also universal with respect to tensor functors 
in an abelian tensor category $\cA$ with a right exact tensor product. Therefore, such an assumption on the category $\cA$ where our cohomology $H$ is taking values has to be satisfied. This sheds some light on the geometric meaning of K\"unneth formulas in the additive case but also raises the question of getting a canonical tensor structure directly on $\cA[\T_H^{op}]$ without making reference to Freyd's universal construction, i.e., not using the equivalence \eqref{end}. 

Finally, all this clearly applies to the construction of effective cohomological Nori motives for several different (new and old) geometric categories. 

\subsubsection*{Acknowledgements} We would like to thank A. Huber for helpful discussions on some of the matters treated herein and the referee for his/her suggestions that considerably improved the exposition.

\section{Good pairs and cellularity} We here assume to be given a base category $\cC$ along with a subcategory $\mathcal{M}$ of distinguished monomorphisms, containing all isomorphisms of $\cC$, and saturated, in the sense that if the composition of a distinguished monomorphism with a morphism is a distinguished monomorphism then the morphism is a distinguished monomorphism.   We consider $\cC^{\square}$, the category whose objects are the arrows in $\mathcal{M}$ and whose arrows are commutative squares in $\cC$. We adopt the conventions of \cite[\S 2.1]{BV} regarding this category.  For example we shall denote by $(X, Y)$ the object of $\cC^{\square}$ which is a monomorphism $Y\to X$ in $\mathcal{M}$.  We also assume the hypotheses of \cite[\S 4.4]{BV} on the subcategory $\cM$. In particular, we assume that the distinguished monomorphisms are stable under direct and inverse images (in the sense of \cite[Remark 4.4.1]{BV}).  Also assume that we have joins $Y\cup Z$ of $\mathcal{M}$-subobjects $Y,Z$ of any object $X$, a strict initial object $\emptyset$ of $\cC$ and that $\emptyset\to X$ is in $\mathcal{M}$ for each object $X$ of $\cC$. The key examples are given by $\cC$ being a category of nice topological spaces or schemes (over a base) and $\mathcal{M}$ the subcategory of closed subspaces or subschemes.

We consider the regular cohomological theory $\T^{op}$ on the signature $\Sigma^{op}$ as defined in \cite[Def. 2.3.3]{BV}. The signature $\Sigma^{op}$ of the theory $\T^{op}$ contains sorts $h^n(X, Y)$ for $n\in \Z$ and $Y\hookrightarrow X\in \mathcal{M}$, function symbols $\square^n : h^n(X', Y')\to h^n(X, Y)$ for all arrows $\square: (X,Y) \to (X', Y')$ in $\cC^{\square}$ and additional function symbols $\partial^n :h^n (Y,Z)\to h^{n+1} (X,Y)$ associated to the morphism $\partial : (Y, Z)\to (X, Y)$ in $\cC^{\square}$ given by any pair of composable arrows $Z\hookrightarrow Y\hookrightarrow X$ in $\mathcal{M}$. Such a theory $\T^{op}$ gives rise to the universal abelian category $\cA [\T^{op}]$ introduced in \eqref{deftmot} above, i.e., the category of constructible effective $\T^{op}$-motives (see \cite[\S 4]{BV}). Note that there is also a homological regular theory $\T$ which yields a universal abelian category $\cA [\T]$ and we actually have a duality equivalence
$$\cA [\T]^{op}\cong \cA [\T^{op}]$$
as proven in \cite[Prop. 4.1.4]{BV}. This means that the abelian category of theoretical motives associated with the cohomology theory is just given by the opposite of that for the homology theory. 

We may further assume the existence of an interval object $I^+$ in $\cC$ and then add the regular axiom of $I^+$-invariance and we still get a regular theory $\T^+$ (see \cite[Def. 2.5.1 \& \S 3.8]{BV}). Adding the axiom of $cd$-exactness, which translates Mayer-Vietoris (see \cite[Def. 2.5.2 \& Lemma 3.8.2]{BV}), also provides a regular theory. Moreover, in order to get that $\cA [\T^{op}]$ is an abelian category we here have to assume that each $h^n(X, Y)$ is an abelian group (see \cite[Lemma 4.1.1 \& Prop. 4.1.3]{BV}). A non additive version of the cohomology theory, which we may also denote $\T^{op}$, is directly obtained by weakening the algebraic structure of $h^n(X, Y)$, e.g., removing the assumption that the group is abelian. In the non additive case the corresponding category $\cA [\T^{op}]$ shall be (Barr) exact only. 

Finally, for any such regular theory $\T'$ obtained by adding or removing regular axioms from $\T^{op}$, to give a $\T'$-cohomology $H$ on $\cC$ with values in an abelian (or just  exact) category $\cA$,  i.e., a $\T'$-model $H$ in $\cA$, is equivalent to giving an exact functor $$r_H: \cA [\T']\to \cA$$ which is the so-called realisation functor (see \cite[Prop. 4.1.3 \& Def. 4.1.5]{BV}). This property says that $\cA [\T']$ is universal with respect to such a $\T'$-models.  In fact, this is essentially an instance of the general theorem \cite[Thm. 6.5]{Bu} on regular theories.

\subsection{The theory of a model}
A model $H$ of the theory $\T^{op}$ in an abelian (or  exact) category $\cA$ (see \cite[\S 3]{BV}), a cohomology $H$ for short, is given by $$\{H^n (X,Y)\}_{n\in\Z}\in \cA$$ a $\Z$-indexed family of objects which is a family of contravariant functors
 $$H: \cC^{\square}\to \cA \hspace*{0.5cm} (X,Y)\leadsto \{H^n (X,Y)\}_{n\in\Z}$$
and, further, $H$ is required to satisfy the exactness axiom involving $\partial$, i.e., for any pair of composable arrows $Z\hookrightarrow Y\hookrightarrow X$ in $\mathcal{M}$, we have the long exact sequence \eqref{longexact} of the triple in $\cA$, which is natural (with respect to $\partial$-cubes, see \cite[\S 3.1 \& (2.1)]{BV}). It follows that $H^n (X,Y)=0$ if $Y\cong X$ is an isomorphism (see \cite[Lemma 3.1.1]{BV}). We shall denote $H^n(X)= H^n(X,\emptyset )$ and note that $H^n(\emptyset) =0$.

Furthermore, we may consider the regular theory $\T_H^{op}$ of the model $H$, i.e.,  the theory obtained from $\T^{op}$ by adding all regular axioms which are satisfied by the specific cohomology $H$; these additional axioms are all properties of $H$ that can be expressed by making use of regular sequents on the signature $\Sigma^{op}$. For example, if $H$ satisfies $I^+$-invariance (in the sense of \cite[Def. 2.5.1]{BV}) then this axiom is automatically included in the theory $\T_H^{op}$ as it can be formulated by the equality of two function symbols. 

Suppose now that excision holds for the cohomology $H$, i.e., for $Y\hookrightarrow  X$ and $Z\hookrightarrow  X$ such that $X = Y\cup Z$ we have that $\triangle : (Y, Y\cap Z)\to (X,Z)$ induces
$$\triangle^n_H : H^n(X,Z)\stackrel{\simeq}{\longrightarrow} H^n(Y, Y\cap Z)$$
an isomorphism in $\cA$ for all $n\in\Z$.  Since we have the function symbol $\triangle^n :  h^n(X,Z)\to h^n(Y, Y\cap Z)$ in the signature $\Sigma^{op}$, excision can be expressed in the language of the theory $\T^{op}$ by the following regular sequents:  $\top\vdash_{x}  \triangle^n (x) = * \rightarrow x = *$ and $\top\vdash_{y} y=y \rightarrow  (\exists x) \triangle^n (x)= y$. The validity of these sequents for a model $H$ is equivalent to the fact that $\triangle^n_H$ is an isomorphism of group objects in the regular category $\cA$ (by using \cite[I.5.11]{Ba}). Therefore, adding these axioms to the theory $\T^{op}$ is guaranteeing that the inverse of $\triangle^n$ exists in any model; conversely, if the cohomology $H$ satisfies excision we then have these sequents in the theory $\T_H^{op}$ of the model $H$.

However, this axiom is not equationally expressible in the theory $\T_H^{op}$; for that we would have to add a function symbol $h^n(Y, Y\cap Z)\to h^n(X,Z)$ corresponding to the inverse of $\triangle^n$ (that is not included in $\Sigma^{op}$).
As usual, if excision holds true for $H$, we obtain the following Mayer-Vietoris exact sequence $$\cdots\to H^n(X) \to H^n(Y)\oplus H^n(Z)\to H^n (Y\cap Z) \to H^{n+1} (X) \to \cdots$$
but, similarly, to express exactness of this sequence as axioms of $\T_H^{op}$ we would have to enlarge the signature 
$\Sigma^{op}$ to include function symbols $h^n (Y\cap Z) \to h^{n+1} (X)$. We invite the interested reader to compare with $cd$-exactness in \cite[Def. 2.5.2]{BV}.

As already stated, setting $\T'=\T_H^{op}$ on the signature $\Sigma^{op}$ we obtain a universal category $\cA [\T_H^{op}]$, according with the notation adopted above. 

Nori's quiver $D^{\rm Nori}$ naturally arises from the sorts and function symbols of the signature $\Sigma^{op}$.
Recall that the quiver $D^{\rm Nori}$ corresponding to $\cC^\square$ is given by the vertices $(X,Y, n)$ for each $n\in\Z$ and edges $(X', Y', n)\to (X, Y, n)$ for each morphism $\square : (X, Y)\to (X', Y')$ in $\cC^\square$  and an additional edge $(Y, Z, n)\to (X, Y, n+1)$ for the morphism $\partial  : (Y,Z) \to (X,Y)$. We then get a representation
$$H: D^{\rm Nori}\to \cA \hspace*{0.5cm} (X,Y, n)\leadsto H^n (X,Y)$$
of Nori's quiver. This representation $H$ of $D^{\rm Nori}$ in  $\cA$ abelian yields a universal abelian category $\cA (H)$ (as indicated in the introduction, see \cite[\S 1]{BVP}) along with the canonical equivalence \eqref{motives} in the additive case. In fact, for any $\T^{op}$-model $H$ in $\cA$
along with its realisation functor $r_H$ we obtain a universal factorisation through the Serre quotient as follows (see \cite[Prop. 4.3.4]{BV} and \cite[Cor. 2.9]{BVP})
$$\xymatrix{\cA [\T^{op}] \ar@{->>}[r] \ar@/^1.7pc/[rr]^-{r_H} & \cA [\T_H^{op}] \ar@{^{(}->}[r]^-{} & \cA}$$
Note that the realisation $r_H$ is faithful exactly if the model $H$ is conservative, meaning that $\T_H^{op}$ and $\T^{op}$ have equivalent categories of models (compare with \cite[Cor. 2.8]{BVP}). This means that if a regular sequent is valid in the regular theory $\T_H^{op}$ of the model $H$ then it is already valid in $\T^{op}$, and conversely.

In particular,  the universal model $U$ in $\cA [\T^{op}]$, whose realisation is the identity functor, is conservative (see also \cite[Prop. 6.4]{Bu}), that is, the regular theory of the universal model is exactly $\T^{op}$ and we have (see \cite[Thm 2.7]{BVP})
$$\catN{U} \cong \cA[\T^{op}].$$

This also shows, as a consequence of the universal representation theorem of \cite{BVP}, that the category $\cA[\T^{op}]$ itself is a Serre quotient of Freyd's free abelian $\catF{D^\Nori}$ on the quiver $D^\Nori$.

\subsection{Good pairs}
Following Nori's original idea, for a cohomology $H: \cC^{\square}\to \cA$ where $(\cA, \tensor )$ we now assume to be a tensor abelian category (in the sense of \cite{DMT}), we consider those pairs $(X, Y)\in \cC^\square$ such that $H^{m}(X, Y)$ is non-zero in a single degree $n$ and zero otherwise. With the same assumptions and notations as in \cite{BVHP}, we shall further assume that $H^{n}(X, Y)$ belongs to a $\flat$-subcategory $\cA^\flat\subset\cA$.  Recall that such a $\flat$-subcategory is a full additive subcategory $\cA^\flat$ of $\cA$ such that all objects of $\cA^\flat$ are flat w.r.t.\,$\tensor$ and $\cA^\flat$ is closed under kernels (see \cite[Def. 1.7]{BVHP}). We may additionally assume that the objects of $\cA^\flat$ are projectives in $\cA$ (and this is coherent with the framework explained in \cite[Rmk. 1.10]{BVHP}) but this is not really needed.
\begin{defn}\label{gp}
Let $(\cA, \tensor)$ be an abelian tensor category, with a right exact tensor $\tensor$. Define the \emph{good pairs} for a model $H$ as above to be those $(X, Y)\in \cC^\square$ such that there exists an integer $n$ such that $H^{m}(X, Y)=0$ for $m\neq n$ and  $H^{n}(X, Y)\in \cA^\flat\subset\cA$. \end{defn}
\begin{defn}\label{gf} An  \emph{$\mathcal{M}$-filtration of dimensional type $d$} of $X\in \cC$ is a finite filtration by $\cM$-subojects $$\Phi: X=X_d \supset \cdots \supset X_p\supset X_{p-1}\supset \cdots \supset X_{-1}=\emptyset$$ 
 where $d\geq -1$ is an integer.
We say that such an $\mathcal{M}$-filtration $\Phi$ is \emph{good} if each $(X_p, X_{p-1})$ is a good pair with $H^{n}(X_p, X_{p-1})=0$ if $n\neq p$.
\end{defn} Clearly we can order (good)  $\mathcal{M}$-filtrations by inclusion at each level. In the same way we can also form the join of two $\mathcal{M}$-filtrations of the same dimensional type obtaining an $\mathcal{M}$-filtration but this is not necessarily preserving that the filtration be good. A dual version of \cite[Lemma 4.4.2]{BV} (without passing to ``lim'') yields the following converging \emph{coniveau spectral sequence} in $\cA$
$${ }^\Phi E_1^{p , q}(X) =  H^{p+q}(X_{p}, X_{p-1})\Rightarrow H^{p+q}(X)$$
 depending on the choosen $\mathcal{M}$-filtration $\Phi$. For a refinement of filtrations $\Phi\subseteq \Psi$ we obtain a morphism of spectral sequences from ${ }^\Psi E_1^{p , q}(X)$ to  ${ }^\Phi E_1^{p , q}(X)$. For $\Phi$ a $\mathcal{M}$-filtration on $X$, $\Phi^\prime$ a $\mathcal{M}$-filtration on $X^\prime$ and $f : X\to X'$ a morphism in $\cC$, which is compatible with the filtrations, i.e., it induces $(X_{p}, X_{p-1})\to (X^\prime_{p}, X^\prime_{p-1})$ in $\cC^\square$, we obtain a map of the spectral sequences ${ }^\Phi E_1^{p , q}(X^\prime) \to { }^{\Phi^\prime} E_1^{p , q}(X)$, the induced morphism on the abutments $H^{p+q}(X')\to H^{p+q}(X)$ being the one given the functoriality.

\subsection{Cellularity} Assume that we have a well defined notion of dimension,  for any object $X\in \cC$ an integer $\dim (X)=d$, and that we have good $\mathcal{M}$-filtrations of the same dimensional type $d$. In the key geometric examples the dimension exists, e.g., the dimension of a scheme of finite type over a field or the dimension of a CW complex. In general, we here simply assume that $X'\subseteq X$ implies $\dim (X')\leq \dim (X)$ and $\dim (\emptyset) = -1$.
\begin{defn}\label{cell}
 Say that the cohomology $H$ is \emph{cellular} with respect to $\cC^\square$ if the following assumptions and conditions are satisfied: {(i)}\, for any object $X\in\cC$ there exists a non-empty family of good  $\mathcal{M}$-filtrations of the same dimensional type $d = \dim (X)$, {(ii)}\, for any two $\mathcal{M}$-filtrations of $X$ there is a third good $\mathcal{M}$-filtration containing the join of the given ones and {(iii)}\, if $f : X\to X'$ is a morphism in $\cC$ then there is a good $\mathcal{M}$-filtration on $X'$ containing the direct image under $f$ of any given good $\mathcal{M}$-filtration of $X$.
 \end{defn}

Now let $\cC^\good\subseteq \cC^\square$ be the full subcategory of good pairs for $H$. We  have that $H$ is a model of $\T^{op}$ for the restricted signature $\Sigma^{op}_\good$ corresponding to $\cC^\good$. We shall denote by $H^\good$ this model.

 Let $D^\good\subset D^{\rm Nori}$ be the full subquiver given by  $\cC^\good$ so that the restriction of $H$ to $D^\good$ is a representation
 $$H^\good: D^\good\to \cA^\flat\subset \cA \hspace*{0.5cm} (X,Y, n)\leadsto H^n (X,Y)$$  in the subcategory $\cA^\flat\subset \cA$ as above. Because of cellularity, we shall see that a dual version of \cite[Lemma 4.4.4]{BV} applies here (without passing to ``lim''). We have that:

\begin{lemma}\label{lemma:goodpair}
If the cohomology $H$ is cellular then $$ \cA[\T_{H^\good}^{op}]\cong \catN{H^\good}\cong \catN{H}\cong \cA[\T_{H}^{op}]$$ are all equivalent abelian categories.
\end{lemma}
\begin{proof} Note that since $D^\good\subset D^{\rm Nori}$ we have that  
${\rm Ab} (D^\good) \hookrightarrow {\rm Ab} (D^{\rm Nori})$ is faithful and we therefore obtain faithful exact functors
$$F^\good :\cA[\T_{H^\good}^{op}]\cong\catN{H^\good}\hookrightarrow \catN{H}\hookrightarrow \cA$$ 
making use of \eqref{motives} for $H^\good$.

We further have that each $X\in \cC$ is provided with a good $\mathcal{M}$-filtration $\Phi$ of dimensional type $d= \dim (X)$ (by Definition \ref{cell} (i)); for any such good $\mathcal{M}$-filtration $\Phi$, $H^{p+q}(X_{p}, X_{p-1})=0$ for $q\neq 0$ and the coniveau spectral sequence ${ }^\Phi E_1^{p , q}(X)$ degenerates in $\cA$. We then have that $H^p(X)\in \cA$ is canonically isomorphic to the $p$-cohomology of the following complex ${ }^\Phi E_1^{* , 0}(X)= { }^\Phi C^{*}_H(X)\in D^b(\cA)$:
$$ \cdots \to H^{p-1}(X_{p-1}, X_{p-2})\to H^{p}(X_{p}, X_{p-1})\to H^{p+1}(X_{p+1}, X_{p})\to \cdots $$ where ${ }^\Phi C^{p}_H(X):= H^{p}(X_{p}, X_{p-1})$. Note that ${ }^\Phi C^{*}_H(X)$ is bounded in $[0, d]$ for $X\neq \emptyset$.   Moreover, for two good $\cM$-filtrations $\Phi$ and $\Phi^\prime$ there is a third good filtration $\Psi$ such that $\Phi \subseteq \Psi$ and $\Phi^\prime\subseteq \Psi$ so that the coniveau spectral sequence ${ }^\Psi E_1^{p , q}(X)$ maps to  both ${ }^\Phi E_1^{p , q}(X)$ and  ${ }^{\Phi^\prime} E_1^{p , q}(X)$ (by Definition \ref{cell} (ii));  therefore, we obtain cochain maps from ${ }^\Psi C^{p}_H(X)$ to both complexes ${ }^\Phi C^{p}_H(X)$ and ${ }^{\Phi^\prime} C^{p}_H(X)$ in such a way that the choice of two good $\mathcal{M}$-filtrations $\Phi$ and $\Phi^\prime$ yields a quasi-isomorphism ${ }^\Phi C^{p}_H(X)\cong { }^{\Phi^\prime} C^{p}_H(X)$ in $D^b(\cA)$.
If $f : Y\to X$ is a morphism and $\Phi$ is a good $\mathcal{M}$-filtration of $Y$ we pick a good $\mathcal{M}$-filtration $\Psi$ on $X$ which is compatible with $f$ (by Definition \ref{cell} (iii)), whence we obtain a cochain map ${ }^\Psi C^{*}_H(X)\to { }^\Phi C^{*}_H(Y)$ showing that the isomorphism $H^n({ }^{\dag}C^{*}_H( - ))\cong H^n( - )$ can be made functorial on $\cC$.  In particular, if $(X,Y)$ is an object of $\cC^\square$ we get $H^p({ }^{\Psi, \Phi}C^{*}_H(X,Y))\in \cA$, where ${ }^{\Psi, \Phi}C^{*}_H(X,Y)$ is the fiber (= the shifted mapping cone) of ${ }^{\Psi}C^{*}_H(X)\to { }^{\Phi}C^{*}_H(Y)$ and 
$H^n({ }^{\dag}C^{*}_H( - ))\cong H^n( - )$ can be made functorial on $\cC^\square$.

Thus the universal model and/or representation in $\cA[\T_{H^\good}]\cong\catN{H^\good}$ can be extended to a $\T_{H}$-model and/or representation of $D^{\rm Nori}$ compatibly with $H^\good$. By the universality we then easily obtain a faithful exact functor $F_H: \cA[\T_{H}]\cong\catN{H}\hookrightarrow \cA[\T_{H^\good}]\cong\catN{H^\good}$ providing a quasi-inverse of $F^\good$.
\end{proof}
\begin{remark}
Observe that the original argument due to Nori (c.f. \cite[Rem. 9.2.5]{HMS}) in the proof that singular cohomology is cellular with respect to $\cC$ being the category of affine schemes over a field $k$ and $\cM$ the subcategory of closed subschemes follows from the so called ``basic Lemma'' (see \cite[\S 2.5]{HMS}). The previous Lemma \ref{lemma:goodpair} is sufficient to detect Nori motives of all algebraic schemes because the inclusion of affine schemes in all algebraic schemes induces an equivalence on the corresponding categories of motives (e.g., see \cite[Thm. 9.2.22]{HMS}). Similarly, the same holds for Voevodsky motives (e.g., see \cite[Cor. 5.16]{CG}).  We leave it to the interested reader to determine general conditions on a cohomology theory $H$ which allow such extension from the affine case. For example, assuming that $I^+$-invariance holds for $H$ (see \cite[Def. 2.5.1 \& \S 3.8]{BV}) and the existence of an affine homotopy replacement one can adapt the arguments in the proof of \cite[Thm. 4.3.2]{FJ}.
\end{remark}

\section{K\"unneth axiom}
Here we express an axiom which corresponds to the usual K\"unneth formula. First of all we need to assume that $\cC$ has finite products $\times$ and a final object $1$ and that the distinguished monomorphisms in $\cM$ are stable under products. Moreover, we shall assume distributivity of joins with respect to products, i.e., for $\cM$-subojects $Y, Z\subset X'$ we assume that $X\times (Y\cup Z) = X\times Y \cup  X\times Z$.  We have that $\cC^\square$ is provided with a product, as follows
$$(X, Y)\times (X', Y') := (X \times X', X \times Y' \cup Y\times X')$$
Also assume that we have a cohomology satisfying excision in an abelian tensor category with a right exact tensor.

\subsection{External product} Fixing  $(\cA, \tensor)$ an abelian tensor category (with a right exact tensor) we are now going to describe a \emph{$\otimes$-model $H$} in $(\cA, \tensor)$ or a cohomology $H$ provided with an \emph{external product $\kappa^H$} by the following data and conditions.  We assume given a cohomology $H :\cC\to \cA$ together with a morphism
$$\kappa_{n,n'}^H : H^n(X,Y)\otimes H^{n'}(X',Y')\to H^{n+n'}(X \times X', X \times Y' \cup Y\times X')$$
for all objects $(X, Y)$ and $(X', Y')$  in $\cC^\square$ and $n, n'\in \Z$.
Note that $\kappa^H$ is providing, as usual, by composition with the diagonal $\Delta : (X, Y\cup Z)\to   (X \times X, X \times Y \cup Z\times X)$ a cup product $$\cup : H^n(X,Y)\otimes H^{m}(X,Z)\to H^{n+m}(X,Y\cup Z)$$ on cohomology and conversely. We shall assume that $\kappa^H$ satisfies the following axioms:
\begin{itemize}
\item[Ax.0] $\kappa_{}^H$ is compatible with the associativity and commutativity constraints given by the product in $\cC^\square$,
\item[Ax.1] $\kappa_{}^H$ is natural in both variables with respect to morphisms in $\cC^\square$,
\item[Ax.2] for $Z\subseteq Y \subseteq X$ and $Y'\subseteq X'$ the following diagram 
\[\begin{xy}\xymatrix{
H^n(Y,Z)\otimes H^{n'}(X',Y')\ar[r]^{\kappa_{}^H\ \ }\ar[d]_{\partial^n\tensor \id}&H^{n+n'}(Y \times X', Y \times Y' \cup Z\times X')\ar[d]^{\delta^{n+n'}}\\
 H^{n+1}(X,Y)\otimes H^{n'}(X',Y')\ar[r]^{\kappa_{}^H\ \ }&H^{n+n'+1}(X \times X', X \times Y' \cup Y\times X')
}\end{xy}\] commutes with sign $(-1)^{n'}$, where $\delta^{k}$ for $k= n+n'$ is the composition of the following three morphisms:
\begin{itemize}
\item[{\it i)}] the morphism $H^{k}(Y \times X', Y \times Y' \cup Z\times X') \to H^{k}(Y \times X', Y \times Y')$ induced by functoriality 
\item[{\it ii)}] the inverse of $\Delta^k$, the excision isomorphism  
$$H^{k}(Y \times X', Y \times Y')\stackrel{\simeq}{\longleftarrow} H^{k}(X \times Y'\cup Y \times X', X \times Y')$$
\item[{\it iii)}] the morphism $\partial^k : H^{k}(X \times Y'\cup Y \times X', X \times Y')\to H^{k+1}(X \times X', X \times Y' \cup Y\times X')$ 
\end{itemize}
\item[Ax.3] and the following commutes
\[\begin{xy}\xymatrix{
H^{n'}(X',Y')\tensor H^n(Y,Z)\ar[r]^{\kappa_{}^H\ \ }\ar[d]_{\id \tensor \partial^n}&H^{n'+n}(X' \times Y ,
X'\times Z \cup Y' \times Y )\ar[d]^{\delta^{n'+n}}\\
H^{n'}(X',Y')\otimes H^{n+1}(X,Y)\ar[r]^{\kappa_{}^H\ \ }&H^{n'+n+1}(X' \times X, X'\times Y\cup Y' \times X).
}\end{xy}\]
where $\delta^{k}$ for $k= n'+n$ is now induced by the composition of the following three morphisms:
\begin{itemize}
\item[{\it i)}] the morphism $H^{k}(X' \times Y ,
X'\times Z \cup Y' \times Y )\to H^{k}(X' \times Y, Y' \times Y)$ induced by functoriality 
\item[{\it ii)}] the inverse of $\Delta^k$, the excision isomorphism  
$$H^{k}(X' \times Y, Y' \times Y)\stackrel{\simeq}{\longleftarrow} H^{k}(X' \times Y\cup Y' \times X, Y' \times X)$$
\item[{\it iii)}] the morphism $\partial^k : H^{k}(X' \times Y\cup Y' \times X, Y' \times X)\to H^{k+1}(X' \times X, X'\times Y\cup Y' \times X)$ 
\end{itemize}

\end{itemize}
We further assume that there is an isomorphism $\varepsilon :H^0(1,\emptyset ) \cong \unit$ with the unit of the tensor structure in such a way that
\begin{itemize}
\item[Ax.4] $H^0(1, \emptyset )\otimes H^n(X, Y)\xrightarrow{\kappa} H^n(1 \times X, 1\times Y\cup \emptyset\times X)\xrightarrow{u^*} H^n(X, Y)$
is the canonical isomorphism $\unit \tensor H^n(X, Y)\cong H^n(X, Y)$ via $\varepsilon$, where $u^*$ is induced by the canonical isomorphism $u: (X, Y) \to (1 \times X, 1\times Y\cup \emptyset\times X)$ (note that here we have that $\emptyset\times - = \emptyset$ as $\emptyset$ is strict initial).
\end{itemize}

\begin{remark} We could consider a theory $\T^\tensor$ which is the extension of the theory $\T^{op}$ on a new signature $\Sigma^\tensor$ containing $\Sigma^{op}$,  such that the $h^n(X, Y)\otimes h^{n'}(X', Y')$ are additional sorts, requiring that they are abelian groups and the $\kappa_{n,n'} : h^n(X,Y)\otimes h^{n'}(X',Y')\to h^{n+n'}(X \times X', X \times Y' \cup Y\times X')$
 are function symbols which are homomorphisms for each $n, n'\in \Z$.

In that case we would have to assume that $h^n(X, Y)\otimes h^{n'}(X', Y')$ are functorial in each variable and also introduce other sorts and function symbols in order to express the above axioms and the assumption that $h^0(1, \emptyset )$ is a unit.  For example, we would need to add function symbols for each variable of $-\tensor +$ corresponding to the function symbols of $\Sigma$, that is, when $\square^n$ is the function symbol associated to a morphism $\square$, or $\partial^n$ for $\partial$, and for the identity function symbol $\boxdot^{m}$ associated to the identity in $\cC^\square$ (refer to \cite{BV} for notation).

Note that we could then also include axioms expressing the strong K\"unneth formula, i.e. that the $``\oplus" \kappa_{i, n-i}$ are isomorphisms, by regular sequents. For example  the surjectivity condition can be easily  expressed by the regular sequent $$\top\vdash_{y} y =y \rightarrow  \exists \overline{x} (\kappa_{0, n}(x_0) + \kappa_{1, n-1}(x_1) + \cdots +\kappa_{n, 0}(x_n)=y).$$

However, few models actually satisfy the strong K\"unneth formula and even in those cases it is not clear how to provide a tensor structure on $\T^{op}$-motives going through $\T^\tensor$-motives. However, it appears to be interesting to understand the differences between the corresponding syntactic categories: the regular syntactic category for $\T^{op}$ on the signature $\Sigma^{op}$ which yields $\cA[\T^{op}]$ and the  exact completion of the syntactic category for $\T^\tensor$ on the extended signature $\Sigma^\tensor$.
\end{remark}
\subsection{Nori's graded $\tensor$-quiver}
For $(X, Y, n)$ and $(X', Y', n')$ in $D^\Nori$ we set
$$(X, Y, n)\otimes (X', Y', n') := (X \times X', X \times Y' \cup Y\times X', n+n')$$
which is a vertex of $D^\Nori$. The grading $| \cdot | : D^\Nori\to \Z_2$ given by  $|(X, Y,n)| =n$ modulo 2 can be considered here as in \cite[\S 2]{BVHP}. We have, following the notation of that reference, the arrows, which we denote $\alpha$, $\beta$, $\beta'$, for expressing the commutativity and associativity constraints, the unit $\unit = (1, \emptyset , 0)$ and arrows $u$ for expressing its properties all given by the canonical choices in $\cC^\square$.
We shall denote the graded Nori $\tensor$-quiver by $D^{\tensor} = (D^{\rm Nori}, \id,\tensor, \alpha,\beta,\beta',\unit,u)$ with the relations listed in \cite[Def. 2.11 \& Def. 2.12]{BVHP}.

\subsection{K\"unneth formula} Let $H$ be a cohomology in an abelian tensor category $\cA$ with an external product $\kappa^H$. Following Nori's original idea we can actually make use of the K\"unneth formulas in order to get a graded $\tensor$-representation of Nori's $\tensor$-subquiver $D^\good$ which consists of the good pairs for the cohomology $H$. We have a canonical comparison map
$$\oplus \kappa_{i, n-i}^H : ``\bigoplus_{i\in \Z}{}" H^i(X,Y)\otimes H^{n-i}(X',Y')\to H^{n}(X \times X', X \times Y' \cup Y\times X')$$
Clearly, for good pairs, we have that only a single pair $(i, n-i)$ of degrees gives a non-zero component in the sum  $``\oplus "\kappa_{i, n-i}^H$.
\begin{defn}\label{kf}
We say that a cohomology $H$ in $\cA$, provided with an external product $\kappa^H$, satisfies the \emph{K\"unneth formula} if, for any $n, n'\in \Z$ and for all  pairs $(X, Y)$ and $(X', Y')$ in $\cC^\good$, we have that
$$\kappa_{n, n'}^H : H^n(X,Y)\otimes H^{n'}(X',Y')\to H^{n+n'}(X \times X', X \times Y' \cup Y\times X')$$
is an isomorphism and
\[\begin{xy}\xymatrix{
H^{n}(X,Y)\otimes H^{n'}(X',Y')\ar[r]^{\kappa_{}^H\ \ }\ar[d]_{}&H^{n+n'}(X \times X', X \times Y' \cup Y\times X')\ar[d]^{\alpha^*}\\
H^{n'}(X',Y')\otimes H^{n}(X,Y)&H^{n'+n}(X' \times X, X'\times Y\cup Y' \times X)\ar[l]^{(\kappa_{}^H)^{-1}\ \ }
}\end{xy}\]
commutes with sign $(-1)^{nn'}$.
\end{defn}

\begin{lemma}\label{lemma:kunneth}
If $H :\cC\to \cA$ is a cohomology provided with an external product $\kappa^H$ that satisfies the K\"unneth formula, then  $D^\good$ is a graded $\tensor$-subquiver of $D^{\rm Nori}$ and $H^\good$ is a graded $\tensor$-representation of  $D^\good$ in $\cA^\flat\subset\cA$ via the external product $\kappa^H$. 
\end{lemma}
\begin{proof} In fact, $(X, Y, n)\otimes (X', Y', n')= (X \times X', X \times Y' \cup Y\times X', n+n')$ is a good pair and $H^\good$ is a graded $\tensor$-representation (see \cite[Def. 2.15]{BVHP}) by construction.
\end{proof}
Let us summarise the assumptions on the category $\cC$ in order to state our main result. Making reference to \cite[\S 4.4]{BV} for the explanation of some terminology we assume that
\begin{itemize}
\item[(a)] the category $\cC$ is provided with a final object $1$, products $X \times X'$ for $X, X'\in \cC$ and a strict initial object $\emptyset \in \cC$;  

\item[(b)] there is a subcategory $\cM$ of distinguished monomorphisms which are stable under products and such that: 
\begin{itemize}
\item[(b.1)] $\cM$ is saturated,  contains all isomorphisms $Y\cong X$ and all $\emptyset\to X$ for $X\in \cC$, 
\item[(b.2)] for $\cM$-subojects $Y\subset X$ and $Y'\subset X'$ in $\cM$ and any morphism $f : X\to X'$ in $\cC$ there is a direct and inverse image, respectively $f_*(Y)\subset X'$ and $f^*(Y')\subset X$ in $\cM$, and 
\item[(b.3)] there are joins $Y\cup Z\subset X$ of $Y\subset X$ and $Z\subset X$ of $\cM$-subojects and they are distributive with respect to products; finally, we assume that
\end{itemize}
\item[(c)] there is a dimension function $\dim : {\rm Ob}\, \cC \to \Z$ such that $\dim  (\emptyset ) = -1$ and $X'\subseteq X$ in $\cM$  implies $\dim (X')\leq \dim (X)$.
\end{itemize}
Note that in the concrete geometric categories $\cC$ endowed with a closure operator the subcategory $\cM$ can be given by the subcategory of closed monomorphisms.

\begin{thm}\label{thm:Tmot} Let $\cC$ be a category and $\cM$ a subcategory, satisfying the assumptions and conditions (a), (b) and (c) here above. Let $H :\cC^{\square}\to \cA$ be a cohomology taking values in an abelian tensor category $\cA$ with a right exact tensor and let $\cA^\flat\subset\cA$ be a $\flat$-subcategory. If $H$ is cellular with respect to $\cC^\square$ and $\cA^\flat$ (see Definition \ref{cell}), and it is provided with an external product and satisfies the K\"unneth formula (see Definition \ref{kf}) then the abelian category of motives ${\sf ECM}_\cC^H$ associated to $\cC$ and $H$ is provided with a canonical tensor structure such that the faithful exact functor ${\sf ECM}_\cC^H\to \cA$ is a tensor functor.
\end{thm}
\begin{proof} According with the notation introduced in \eqref{Nori} we here just use $\catN{H^\good} \cong \cA[\T_H]$ by Lemma \ref{lemma:goodpair}. In fact, Lemma \ref{lemma:kunneth} yields that the subdiagram $D^\good\subset D^{\rm Nori}$ given by good pairs along with $H^\good: D^\good\to \cA^\flat$ satisfies the assumptions of the universal graded $\tensor$-representation theorem \cite[Thm 2.18]{BVHP}.
\end{proof}

\subsection{Applications}
All this clearly applies to the well known case of $H$ being cellular cohomology on the category $\cC$ of CW complexes, canonically filtered by $n$-skeletons. Moreover, it applies to $H$ being singular cohomology on the category $\cC$ of affine algebraic $k$-schemes for $k$ a subfield of the complex numbers $\C$; the singular cohomology is cellular because of Nori's basic lemma and we can apply Theorem \ref{thm:Tmot} where $\cA$ is the category of abelian groups and where $\cA^\flat$ is the subcategory of free abelian groups (see also \cite[Thm. 4.5]{BVHP}).  Similarly, the interested reader can see that our formalism applies to \cite{Ar}. Note that the representation of singular cohomology in the abelian tensor category $\cA = {\rm MHS}$ of mixed Hodge structures yields back Nori's original category: as soon as we consider the relative case, considering Betti representation in Saito's mixed Hodge modules $\cA = {\rm MHM}$ we get Ivorra's perverse Nori motives \cite{I1}.

Following \cite{FJ}, for pairs $(X, f)$ where $X$ is an algebraic variety defined over a subfield $k$ of $\C$ and $f:X\to \C$ a regular function consider $H^n(X, f)$ the rapid decay cohomology; for $\cC$ the category of pairs $(X, f)$ with $X$ affine this is yielding a $\T^{op}$-model $H$ in the category $\cA$ of finite dimensional $\Q$-vector spaces; one has an exponential basic lemma and a K\"unneth formula (see \cite[Chap. 3-4]{FJ}) in such a way that our assumptions in Theorem \ref{thm:Tmot} are satisfied by rapid decay cohomology, yielding the desired tensor structure on exponential motives. Similarly, it appears that our construction applies in the context of hypergeometric motives \cite{NP} as well. 

Finally, as a conjectural application, one aims at constructing a ``sharp" singular cohomology $H$ on a suitable category $\cC$ based on algebraic varieties (e.g., see \cite[\S 0.2]{BB} for more details) and where now $\cA =$ FHM, formal Hodge structures, or EHM, enriched Hodge structures, or, lastly, MHSM, mixed Hodge structures with modulus (see \cite{IY} for work in this direction). One seeks a ``sharp" singular cohomology $H$ that satisfies the hypotheses in Theorem \ref{thm:Tmot} in such a way that we would get an abelian tensor category of ``sharp" Nori motives (or motives with modulus).

\end{document}